\documentclass{elsarticle}

\usepackage[english]{babel}
\usepackage{dsfont} 
\usepackage{graphicx}
\usepackage{color}
\usepackage[hypertexnames=false]{hyperref} 
\usepackage{enumerate}
\usepackage{amssymb,empheq} 
\usepackage{amsthm} 
\usepackage{thmtools} 
\usepackage{cleveref}
\usepackage[dvipsnames]{xcolor}
\usepackage{subcaption}

\allowdisplaybreaks
\newtheorem{thm}{Theorem}[section]
\newtheorem{lem}[thm]{Lemma}
\newtheorem{prop}[thm]{Proposition}
\newtheorem{cor}[thm]{Corollary}

\theoremstyle{definition}
\newtheorem{definition}[thm]{Definition}

\newtheorem{dfn}[thm]{Definition}

\newtheorem{exm}[thm]{Example}

\theoremstyle{remark}
\newtheorem{rem}[thm]{Remark}

\newcommand{\exmsymbol}{\hfill$\circ$}

\newcommand{\cset}{\mathds{C}}

\newcommand{\nset}{\mathds{N}}

\newcommand{\rset}{\mathds{R}}

\newcommand{\conv}{\mathrm{conv}\,}

\newcommand{\diff}{\mathrm{d}}

\newcommand{\pos}{\mathrm{Pos}}

\newcommand{\tr}{\mathrm{tr}}

\newcommand{\id}{\mathrm{id}}
\newcommand{\one}{\mathds{1}}

\newcommand{\Herm}{\mathrm{Herm}}

\newcommand{\Pos}{\pos}
\newcommand{\skl}[2]{\left\langle #1, #2 \right\rangle}

\newcommand{\cA}{\mathcal{A}}

\newcommand{\cat}{\mathcal{C}}

\newcommand{\cL}{\mathcal{L}}

\newcommand{\cX}{\mathcal{X}}

\newcommand{\fB}{\mathfrak{B}}

\author{Philipp J.\ di Dio\footnote{Corresponding author, email: philipp.didio@uni-konstanz.de}$^{\text{,a,b,c}}$ and Lars-Luca Langer\footnote{Email: lars-luca.langer@uni-konstanz.de}$^{\text{,a,c}}$}
\address{$^\text{a}$Department of Mathematics and Statistics, University of Konstanz, Universit\"atsstra{\ss}e 10, D-78464 Konstanz, Germany}
\address{$^\text{b}$Department of Computer and Information Science, University of Konstanz, Universit\"atsstra{\ss}e 10, D-78464 Konstanz, Germany}
\address{$^\text{c}$Zukunftskolleg, Universtity of Konstanz, Universit\"atsstra{\ss}e 10, D-78464 Konstanz, Germany}

\journal{arXiv}

\title{Characterization of Matrix $K$-Positivity Preserver for $K=\rset^n$ and for Compact Sets $K\subseteq\rset^n$}

\begin{document}

\begin{abstract}
For any closed $K\subseteq\rset^n$, in [P.\ J.\ di\,Dio, K.\ Schmüdgen: \textit{$K$-Positivity Preserver and their Generators}, SIAM J.\ Appl.\ Algebra Geom.\ \textbf{9} (2025), 794--824] all $K$-positivity preserver have been characterized, i.e., all linear maps $T:\rset[x_1,\dots,x_n]\to\rset[x_1,\dots,x_n]$ such that $Tp\geq 0$ on $K$ for all $p\geq 0$ on $K$.
An important extension of polynomials $\rset[x_1,\dots,x_n]$ with real coefficients are polynomials $\rset^{m\times m}[x_1,\dots,x_n]$ with matrix coefficients.
Non-negativity on $K$ for matrix polynomials with Hermitian coefficients $\Herm_m$ is then $p(x)\succeq 0$ for all $x\in K$.
In the current work, we investigate linear maps $T:\Herm_m[x_1,\dots,x_n]\to\Herm_m[x_1,\dots,x_n]$.
We focus on matrix $K$-positivity preserver, i.e., $Tp\succeq 0$ on $K$ for all $p\succeq 0$ on $K$.
For $K=\rset^n$ and compact sets $K\subseteq\rset^n$, we give characterizations of matrix $K$-positivity preservers.
We discuss the difference between the real and the matrix coefficient case and where our proof fails for general sets $K\subseteq\rset^n$ with $K\neq \rset^n$ and $K$ non-compact.
\end{abstract}


\maketitle

\section{Introduction}
\label{sec:intro}

Let $n \in \nset$ and denote by $\rset[x_1, \dots, x_n]$ the set of polynomials in $n$ variables with real coefficients.
Let $T: \rset[x_1, \dots, x_n] \to \rset[x_1, \dots, x_n]$ be linear. 
It is long known that every $T$ has a unique representation
\begin{equation}\label{eq:DiffSumReprNonMatrixOp}
T = \sum_{\alpha \in \nset_0^n} q_\alpha\cdot \partial^\alpha
\end{equation}
with unique $q_\alpha\in\rset[x_1, \dots, x_n]$ for all $\alpha \in \nset_0^n$, see e.g.\ \cite[Lem.\ 2.3]{netzer10}.
If $Tp\geq 0$ on $K\subseteq\rset^n$ for any $p\in\rset[x_1,\dots,x_n]$ with $p\geq 0$ on $K$, then $T$ is said to be a $K$-positivity preserver.
$K$-positivity preservers were characterized in \cite[Thm.\ 3.1]{borcea11} for $K=\rset^n$ and in \cite[Thm.\ 3.5]{didio25KPosPresGen} for all closed $K\subseteq\rset^n$.
The aim of this work is to extend these characterization to linear operators on matrix polynomials $\rset^{m\times m}[x_1,\dots,x_n]$.

This paper is structured as follows. 
In the next section (\Cref{sec:prelim}) we collect preliminaries which are needed to prove the main results to make the paper as self contained as possible.
In \Cref{sec:NetzerCharacterizationLinOp} we give a characterization of linear operators
\[T: V[x_1, \dots, x_n] \to V[x_1, \dots, x_n]\]
similar to (\ref{eq:DiffSumReprNonMatrixOp}).
In \Cref{sec:BorceaCharacterizationOnRn} we give our main result (\Cref{thm:BorceaCharacterizationPolynomialMatrices}) which is a generalization of \cite[Thm.\ 3.1]{borcea11} to the polynomial matrix case.
In \Cref{exm:QAlphaMeasureNotPositive} we will explain the problem of writing the coefficients $Q_\alpha$ from \Cref{lem:DiffSumRepresentationRLinearOp} as moment sequences.
In \Cref{sec:BorceaCharacterizationOnCompactSets} we restrict ourselves to compact sets $K\subseteq\rset^n$ where these problems do not occur.

\section{Preliminaries}
\label{sec:prelim}

\subsection{Hermitian Matrices}

Let $m, n \in \nset$, let $\Herm_m \subseteq \cset^{m \times m}$ be the set of all $m \times m$ Hermitian matrices, let $\Herm_m[x_1,\dots,x_n]$ be the set of all polynomials in $n$ variables $x_1,\dots,x_n$ with Hermitian matrix coefficients, and denote by $\Herm_{m,+}$ the subset of positive semi-definite matrices.
The set $\Herm_{m,+} \subseteq \Herm_m$ is a full-dimensional convex cone and hence every $M \in \Herm_m$ can be written as $M = M^+ - M^-$ with $M^+, M^- \in \Herm_{m,+}$. 
A matrix polynomial $P \in \Herm_m[x_1, \dots, x_n]$ is positive on $K\subseteq\rset^n$, written $P \succeq 0$ on $K$, if $P(x) \in \Herm_{m,+}$ for all $x \in K$.
We define
\[\Pos(K) := \big\{ P \in \Herm_m[x_1,\dots, x_n] \;\big|\; P(x) \succeq 0\ \text{ for all } x \in K \big\}.\]

The space of Hermitian matrices is a finite-dimensional real Hilbert space, see e.g.\ \cite[Sect.\ 2]{maedschmued24_01}.
Let $E_{k,l} := (\delta_{i,k}\cdot\delta_{j,l})_{i,j=1}^m$ where $\delta_{i,j}$ is the Kronecker delta function.
Then, for $k,l=1,\dots,m$, the
\[
	H_{k, l} := \begin{cases}
		\frac{1}{\sqrt{2}}(E_{k, l} + E_{l, k}) 
		&\text{ for } k < l,
	\\	E_{k, k} 
		&\text{ for } k = l,
	\\	\frac{i}{\sqrt{2}}(E_{k, l} - E_{l, k}) 
		&\text{ for } k > l.
	\end{cases}
\]
form an orthonormal basis of $\Herm_m$.
The inner product $\langle\,\cdot\,,\,\cdot\,\rangle$ on $\Herm_m$ is given by
\[
	\skl{A}{B} := \tr(A^*B) = \tr(AB) \in \rset
\]
for all $A, B \in\Herm_m$.
Positive semi-definiteness can be characterized by the following.

\begin{lem}[see e.g.\ \protect{\cite[Prop.\ A.21]{schmudMomentBook}}]
\label{lem:MatrixPositiveIffInnerProductWithAllPositiveMatricesIsPositive}
Let $m \in \nset$ and let $M \in \Herm_m$. 
Then the following are equivalent:
\begin{enumerate}[(i)]
\item $M \succeq 0$.
\item $\skl{M}{X} \geq 0$ for all $X \in \Herm_{m,+}$.
\end{enumerate}
\end{lem}

\subsection{Hermitian Matrix Valued Measures}

\begin{dfn}\label{def:MatrixValuedMeasure}
Let $m \in \nset$ and $(\cX,\cA)$ be a measurable space. 
We say $\mu$ is a \emph{matrix-valued measure} on $(\cX,\cA)$ if
\[\mu = (\mu_{i, j})_{i, j = 1, \dots, m}\]
for some finite signed complex measures $\mu_{i, j}$ on $(\cX,\cA)$ and $\mu(S) \in \Herm_m$ for all $S \in \cA$.
A matrix-valued measure $\mu$ is called \emph{positive} if $\mu(S) \in \Herm_+$ for all $S \in \cA$.
\end{dfn}

The definition of positive matrix-valued measures coincides with the definition of positive operator-valued measures from \cite[Rem.\ 2]{cimpzalar13} and \cite[Def.\ 1]{berb09}.
Integrals with respect to these measures are defined in the following way.

\begin{dfn}\label{def:integral_polymatrix}
Let $m, n \in \nset$, $K \subseteq \rset^n$ be closed, $p \in \Herm_m[x_1, \dots, x_n]$, and let $\mu$ be a matrix-valued measure. 
We define the integral
\[
	\int_K \skl{p(x)}{\diff\mu(x)} := \sum_{i, j = 1}^m \int_K p_{i, j}(x) ~\diff\mu_{i, j}'(x) \in \rset
\]
with $p_{i, j} = \skl{p}{H_{i, j}} \in \rset[x_1, \dots, x_n]$ and $\mu_{i, j}'= \skl{\mu}{H_{i, j}}$ a signed measure.
\end{dfn}

Let $m \in\nset$ and $A = (a_{k, l})_{k, l = 1}^m\in\Herm_{m,+}$.
Then
\[
	|a_{k, l}|^2 \leq a_{k, k} \cdot a_{l, l} \leq \tr(A)^2
\]
for all $k, l = 1, \dots, m$. 
This extends to positive matrix-valued measures.

\begin{lem}[see \protect{\cite[Lem.\ 1.1]{schmued87}}]
Let $m \in \nset$, let $\cX$ be a locally compact Hausdorff space equipped with some $\sigma$-algebra $\cA$. 
Let $\mu$ on $(\cX,\cA)$ be a positive matrix-valued measure with components $\mu_{i,j}$, $i,j = 1, \dots, m$.
We define the trace measure $\tr(\mu)$ by
\[\tr(\mu) := \sum_{i = 1}^m \mu_{i, i}.\]
Then $\tr(\mu)$ is a real-valued positive measure and $\mu_{i, j} \ll \tr(\mu)$, i.e., $\mu_{i,j}$ is absolutely continuous with respect to $\tr(\mu)$ for all $i, j = 1, \dots, m$.
\end{lem}

\begin{rem}[see \protect{\cite[Lem.\ 1.1 and 1.2]{schmued87}}, \protect{\cite[Section 2]{maedschmued24_01}}]\label{rem:PositiveMatrixValuedMeasureRadonNikodymDerivative}
Using the Radon--Nikodym derivative
\[\varphi_{i, j}: \rset^n \to \cset \quad\text{with}\quad \varphi_{i, j} := \frac{\diff\mu_{i, j}}{\diff\tr(\mu)},\]
for positive matrix-valued measures $\mu$, the integral from \Cref{def:integral_polymatrix} becomes
\[
\int_K \skl{p(x)}{\diff\mu(x)} = \int_K \skl{p(x)}{\rho(x)}~\diff\tr(\mu)(x)
\]
where
\[\rho: \rset^n \to \cset^{m \times m} \quad\text{with}\quad \rho(x) := \sum_{i, j = 1}^m \varphi_{i, j}(x) E_{i, j}\]
is positive semi-definite and Hermitian $\tr(\mu)$-almost everywhere.
Hence,
\[
\int_K \skl{p(x)}{\diff\mu(x)} \geq 0
\]
for positive matrix-valued measures $\mu$ and $p\in\pos(\rset^n)$.
\exmsymbol
\end{rem}

Another useful property of the integral from \Cref{def:integral_polymatrix} is the following.

\begin{lem}\label{lem:IntegralInnerProductSwap}
Let $m, n \in \nset$ and let $K \subseteq \rset^n$ be closed.
Let $\mu$ be a matrix-valued measure on $K$. 
Let $p \in \rset[x_1, \dots, x_n]$ and $M \in \Herm_m$. 
Then
\[
	\int_K \skl{M\cdot p(x)}{\diff\mu(x)} = \skl{\int_K p(x) ~\diff\mu(x)}{M}.
\]
\end{lem}
\begin{proof}
By $\mu = \sum_{i, j = 1} H_{i, j} \cdot \mu_{i, j}'$ with $\mu_{i, j}' := \skl{\mu}{H_{i, j}}$, we obtain
\begin{align*}
	\int_K \skl{Mp(x)}{\diff\mu(x)}
	&= \sum_{i, j = 1}^m \int_K \skl{Mp(x)}{H_{i, j}} ~\diff\!\skl{\mu}{H_{i, j}}(x)
\\	&= \sum_{i, j = 1}^m \skl{M}{H_{i, j}} \int_K p(x) ~\diff\mu_{i, j}'(x)
\\	&= \skl{M}{\sum_{i, j = 1}^m \int_K p(x) ~\diff\mu_{i, j}'(x) H_{i, j}}
\\	&= \skl{M}{\int_K p(x) ~\diff\mu(x)}.\qedhere
\end{align*}
\end{proof}

We will use the following matrix version of Haviland's Theorem.

\begin{thm}[see \protect{\cite[Thm.\ 3]{cimpzalar13}}]
\label{thm:PositiveLambdaFunctionalExistsPositiveMeasure}
Let $m, n \in \nset$, $K \subseteq \rset^n$ be closed, and
let
\[L: \Herm_m[x_1, \dots, x_n] \to \rset\]
be a linear functional.
Then the following are equivalent:
\begin{enumerate}[(i)]
\item $L$ is a $K$-moment functional, i.e., there exists a positive matrix-valued measure $\mu$ on $K$ such that 
\[
	L(p) = \int_K \skl{p(x)}{\diff\mu(x)}
\]
for all $p \in \Herm_m[x_1, \dots, x_n]$.

\item $L(p) \geq 0$ for all $p \in \pos(K)$.
\end{enumerate}
\end{thm}

The proof of \Cref{thm:PositiveLambdaFunctionalExistsPositiveMeasure} in \cite{cimpzalar13} is analog to the proof of Haviland's theorem. 
Using the M.\ Riesz Extension Theorem, the functional $L$ can be extended to a space containing continuous functions with compact support and by a version of the Riesz--Markov--Kakutani Theorem, this gives a suitable measure. 
The difference to the case with real coefficients is, that for Hermitian matrix coefficients the spaces will be tensor products. 
$\Herm_m[x_1, \dots, x_n]$ for example, can be written as a tensor product $\Herm_m[x_1, \dots, x_n] = \Herm_m \otimes \rset[x_1, \dots, x_n]$. 
We will look deeper into the theory behind this in \Cref{sec:TensorProductCones}.

If we not only work with
\begin{align*}
L: \Herm_m[x_1, \dots, x_n] &\to \rset
\intertext{but with}
L: \Herm_m[x_1, \dots, x_n] &\to \Herm_m
\end{align*}
in \Cref{sec:BorceaCharacterizationOnCompactSets}, then we need the following.

\begin{thm}[see \protect{\cite[Thm.\ 4 and 5]{cimpzalar13}}]\label{thm:PositiveMapCompactSetExistsPositiveMeasure}
Let $m, n \in \nset$, $K \subseteq \rset^n$ be compact, and let
\[L: \Herm_m[x_1, \dots, x_n] \to \Herm_m\]
be linear.
Then the following are equivalent:
\begin{enumerate}[(i)]
\item $L$ preserves positivity on $K$, i.e., $L(\Pos(K)) \succeq 0$.

\item There exists a unique positive Borel operator-valued measure
\[\mu: \fB(K) \to \cL(\Herm_m, \Herm_m)\]
on $K$ such that 
\[L(p) = \int_K p(x) ~\diff\mu(x)\]
for all $p \in \Herm_m[x_1, \dots, x_n]$.
\end{enumerate}
\end{thm}

In the next subsection we want to discuss the M.\ Riesz Extension Theorem on tensor products such as the polynomial matrices
\[\Herm_m[x_1, \dots, x_n] = \Herm_m \otimes \rset[x_1, \dots, x_n],\]
which is used to prove \Cref{thm:PositiveLambdaFunctionalExistsPositiveMeasure} in \cite{cimpzalar13}.
The authors of \cite{cimpzalar13} use cone extensions on tensor products but only shortly mention the details.
Hence, because of the importance of the result, we want to look at it in more detail.

\subsection{Convex Cones on Tensor Products}
\label{sec:TensorProductCones}

\begin{definition}
Let $V_1, V_2$ be two real vector spaces and let $C_1 \subseteq V_1, C_2 \subseteq V_2$ be two convex cones. 
Define the \emph{minimal tensor product} as 
\[
C_1 \otimes C_2 = \conv\{c_1 \otimes c_2 \mid c_1 \in C_1, c_2 \in C_2\}.
\]
\end{definition}
The tensor product of convex cones $C_1 \otimes C_2$ then gives a convex cone in the tensor product of the overlying vector spaces $V_1 \otimes V_2$. 
Then
\[\Herm_{m,+} \subseteq \Herm_m \quad\text{and}\quad \rset[x_1, \dots, x_n]_+ \subseteq \rset[x_1, \dots, x_n]\]
are convex cones.
Hence,
\[\Herm_{m,+} \otimes \rset[x_1, \dots, x_n]_+\]
is a convex cone in
\[\Herm_m \otimes \rset[x_1, \dots, x_n] = \Herm_m[x_1, \dots, x_n].\]
%
Then
\[\Herm_{m,+} \otimes \rset[x_1, \dots, x_n]_+ \subseteq \Pos(\rset^n).\]

For the M.\ Riesz Extension Theorem, we need the following lemma.

\begin{lem}\label{lem:CofinalityExtendsToTensorProduct}
Let $V, F$ be real vector spaces and let $C \subseteq F$ be a convex cone.
Let $E \subseteq F$ be a linear subspace with $F = E + C$ and let $V_+ \subseteq V$ be a full-dimensional cone.
Then 
\[
V \otimes F = (V \otimes E) + (V_+ \otimes C).
\]
\end{lem}
\begin{proof}
Since $V_+ \subseteq V$ is a full-dimensional cone, there exists a basis $\{M_i\}_{i \in I}$ of $V$ for some Index set $I$ such that $M_i \in V_+$ for all $i \in I$. 
Now let $f \in V \otimes F$.
Then there exist
\[k \in \nset,\quad \widetilde{M}_1, \dots, \widetilde{M}_k \in \{M_i\}_{i \in I},\quad\text{and}\quad f_1, \dots, f_k \in F\]
such that 
\[
f = \widetilde{M}_1 \otimes f_1 + \dots + \widetilde{M}_k \otimes f_k.
\]
Now, for $f_1, \dots, f_k \in F$, there exist $e_1, \dots, e_k \in E$ and $c_1, \dots, c_k \in C$ such that 
\[
f_i = e_i + c_i
\]
for all $i = 1, \dots, k$. 
Hence,
\begin{align*}
	f
	&= \widetilde{M}_1 \otimes f_1 + \dots + \widetilde{M}_k \otimes f_k
\\	&= \widetilde{M}_1 \otimes e_1 + \dots + \widetilde{M}_k \otimes e_k
	+ \widetilde{M}_1 \otimes c_1 + \dots + \widetilde{M}_k \otimes c_k
\\	&\in V \otimes E + V_+ \otimes C.
\end{align*}
Since $f \in V \otimes F$ was chosen arbitrarily,
\[V \otimes F \subseteq V \otimes E + V_+ \otimes C.\]
The reverse inclusion holds trivially, since $V_+ \subseteq V$ and $E, C \subseteq F$.
Therefore,
\[
V \otimes F = (V \otimes E) + (V_+ \otimes C).\qedhere
\]
\end{proof}

Note that this lemma states that if $E \subseteq F$ is a cofinal subspace and $V_+ \subseteq V$ is full-dimensional, then $V \otimes E$ is cofinal in $V \otimes F$ with positivity cone $V_+ \otimes F_+$.
The \emph{M.\ Riesz Extension Theorem} now states that every linear functional
\[L: V \otimes E \to \rset\]
which is positive on $(V \otimes E) \cap (V_+ \otimes F_+)$ can be extended to a linear functional
\[\widetilde{L}: V \otimes F \to \rset\]
which is positive on $V_+ \otimes F_+$.

Returning to the polynomial matrix case, let $n, m \in \nset$ and define 
\[
F := \{f \in \cat(\rset^n) \mid \exists p \in \rset[x_1, \dots, x_n]: |f| \leq p\ \text{on}\ \rset^n\}
\]
as well as 
\[
F_+ := \{f \in F \mid f \geq 0\ \text{on}\ \rset^n\}.
\]
Then the continuous functions with compact support $\cat_c(\rset^n)$ are contained in $F$ and 
\[
F = \rset[x_1, \dots, x_n] + F_+.
\]
By \Cref{lem:CofinalityExtendsToTensorProduct},
\[
\Herm_m \otimes F = (\Herm_m \otimes \rset[x_1, \dots, x_n]) + (\Herm_{m,+} \otimes F_+).
\]
Notice that 
\[
	\Herm_{m,+} \otimes \rset[x_1, \dots, x_n]_+ \subseteq (\Herm_m \otimes \rset[x_1, \dots, x_n]) \cap (\Herm_{m,+} \otimes F_+)
\]
as well as 
\[
	(\Herm_m \otimes \rset[x_1, \dots, x_n]) \cap (\Herm_{m,+} \otimes F_+) \subseteq \Pos(\rset^n).
\]
The latter is simply because elements $P \in \Herm_m[x_1, \dots, x_n] = \Herm_m \otimes \rset[x_1, \dots, x_n]$ are matrices with polynomials as entries and elements $f \in \Herm_{m,+} \otimes F_+$ are positive semi-definite at every point on $\rset^n$. 
Hence, every function in the intersection is a polynomial matrix which is positive semi-definite on $\rset^n$, which is exactly the definition of $\Pos(\rset^n)$.
Now every linear functional
\[L: \Herm_m[x_1, \dots, x_n] \to \rset\]
such that $L$ is positive on $\pos(\rset^n)$ can be extended to a linear functional
\[\widetilde{L}: \Herm_m \otimes F\]
such that $\widetilde{L}$ is positive on $\Herm_{m,+} \otimes F_+$. 
With $\cat_c(\rset^n) \subseteq F$ and by (a version of) the Riesz--Markov--Kakutani Theorem, there exists a measure which represents $\widetilde{L}$. 
This is how \Cref{thm:PositiveLambdaFunctionalExistsPositiveMeasure} is proven in \cite{cimpzalar13}.

\section{The Canonical Representation of Linear Operators on Matrix Polynomials}
\label{sec:NetzerCharacterizationLinOp}

At first we give an extension of the canonical representation (\ref{eq:DiffSumReprNonMatrixOp}) to any linear operator $T:V[x_1,\dots,x_n]\to V[x_1,\dots,x_n]$ for an arbitrary real vector space $V$.

\begin{lem}\label{lem:DiffSumRepresentationRLinearOp}
Let $n \in \nset$ and let $V$ be a real vector space. 
Let
\[T: V[x_1, \dots, x_n] \to V[x_1, \dots, x_n]\]
be a linear operator, i.e., $T(af+g) = a\cdot Tf + Tg$ for all $a\in\rset$ and $f,g\in V[x_1, \dots, x_n]$. 
Then, for all $\alpha \in \nset_0^n$, there exist unique linear operators
\[Q_\alpha: V \to V[x_1, \dots, x_n]\]
such that
\begin{equation}\label{eq:DiffSumRepresentationQalpha}
T = \sum_{\alpha \in \nset_0^n} \frac{1}{\alpha!}\cdot Q_\alpha \times \partial^\alpha
\quad\text{with}\quad
(Q_\alpha \times \partial^\alpha)(v\cdot x^\beta) := Q_\alpha(v)\cdot\partial^\alpha x^\beta
\end{equation}
for all $\alpha, \beta \in \nset_0^n$ and $v \in V$.
\end{lem}

The representation (\ref{eq:DiffSumRepresentationQalpha}) is called \emph{canonical representation}.
The proof is a straight-forward extension of the proof in \cite[Lem.\ 2.3]{netzer10}.

\begin{proof}
We will prove \Cref{lem:DiffSumRepresentationRLinearOp} using induction over $d = |\alpha| \in \nset_0$ such that
\begin{equation}\label{eq:NetzerProofRecursionFormula}
T(vx^\beta) = \sum_{\alpha \in \nset_0^n} \frac{1}{\alpha!}\cdot (Q_\alpha \times \partial^\alpha)(vx^\beta) 
= \sum_{\alpha \preceq \beta} \binom{\beta}{\alpha}\cdot Q_\alpha(v)\cdot x^{\beta - \alpha}
\end{equation}
for all $\beta \in \nset_0^n$ with $|\beta| \leq d$ and $v\in V$.

\underline{$d = 0$:} 
For $\beta = 0$, clearly,
\[Q_0:V\to V[x_1,\dots,x_n],\quad v\mapsto Q_0(v) := T(v\cdot 1)\]
is the unique linear $Q_0$ in (\ref{eq:NetzerProofRecursionFormula}).

\underline{$d \to d + 1$:}
Assume, for all $\alpha\in\nset_0^n$ with $|\alpha| \leq d$, there exist unique linear $Q_\alpha:V\to \in V[x_1, \dots, x_n]$ in (\ref{eq:NetzerProofRecursionFormula}), i.e., for all $|\beta| \leq d$.
Now let $\gamma \in \nset_0^n$ with $|\gamma| = d+1$. 
According to (\ref{eq:NetzerProofRecursionFormula}), we require $Q_\gamma:V\to\in V[x_1, \dots, x_n]$ to fulfill
\[
T(vx^\gamma) = Q_\gamma(v) + \sum_{\substack{\alpha \preceq \gamma \\ \alpha \neq \gamma}} \binom{\gamma}{\alpha}\cdot Q_\alpha(v)\cdot x^{\gamma-\alpha}
\]
for all $v\in V$.
Hence,
\[Q_\gamma:V\to V[x_1,\dots,x_n],\quad v\mapsto Q_\gamma(v) := T(vx^\gamma) - \sum_{\substack{\alpha \preceq \gamma \\ \alpha \neq \gamma}} \binom{\gamma}{\alpha}\cdot Q_\alpha(v)\cdot x^{\gamma-\alpha}\]
is unique.
\end{proof}

\begin{rem}
The $Q_\alpha$ for $\alpha\in\nset_0^n$ from \Cref{lem:DiffSumRepresentationRLinearOp} are constructed recursively. 
Using the binomial transform, we get the explicit form
\[
Q_\beta(v) = \sum_{\alpha \preceq \beta} \binom{\beta}{\alpha}\cdot (-1)^{\beta-\alpha}\cdot T(vx^\alpha)\cdot x^{\beta-\alpha}
\]
for all $\beta \in \nset_0^n$ and $v \in V$.
\exmsymbol
\end{rem}

We get the following immediate consequence.

\begin{cor}\label{cor:TreprUnitalAlegbra}
Let $n\in\nset$, let $A$ be a unitial algebra over $\rset$, and let
\[T:A[x_1,\dots,x_n]\to A[x_1,\dots,x_n]\]
be a $\rset$-linear map with
\[T(p\cdot a) = (Tp)\cdot a\]
for all $p\in A[x_1,\dots,x_n]$ and $a\in A$.
Then, for every $\alpha\in\nset_0^n$, there exists a unique $q_\alpha\in A[x_1,\dots,x_n]$ such that
\[T = \sum_{\alpha\in\nset_0^n} \frac{1}{\alpha!}\cdot q_\alpha\cdot\partial^\alpha.\]
\end{cor}
\begin{proof}
By \Cref{lem:DiffSumRepresentationRLinearOp} and $T(p\cdot a) = (Tp)\cdot a$ for all $p\in A[x_1,\dots,x_n]$ and $a\in A$, $Q_\alpha(a) = Q_\alpha(1\cdot a) = Q_\alpha(1)\cdot a$ for all $a\in A$ and $q_\alpha := Q_\alpha(1)\in A[x_1,\dots,x_n]$.
\end{proof}

The following example shows, that when $V=A$ is \emph{not} a unitial algebra, then in general the operators $Q_\alpha$ in \Cref{lem:DiffSumRepresentationRLinearOp} are \emph{not} multiplication by polynomials $q_\alpha\in A[x_1,\dots,x_n]$ like in \Cref{cor:TreprUnitalAlegbra}.
This happens already for $V = \Herm_2$.

\begin{exm}
Let $n=1$ and let
\[T: \Herm_2[x]\to\Herm_2[x],\quad \begin{pmatrix} f & g\\ \overline{g} & h\end{pmatrix}\mapsto \begin{pmatrix}h & g\\ \overline{g} & f\end{pmatrix}.\]
Then $T = Q_0\times \partial^0$, since $T(p \cdot x^k) = (Tp)\cdot x^k$ for all $p\in\Herm_2[x]$ and $k\in\nset_0$.
But
\[Q_0\!\!\left( \begin{pmatrix} a & b\\ \overline{b} & c\end{pmatrix}\right) = \begin{pmatrix}
c & b\\ \overline{b} & a\end{pmatrix} \]
for all $a,c\in\rset$ and $b\in\cset$ implies that there exists no $q_0\in\Herm_2[x]$ such that
\[q_0\cdot \begin{pmatrix} a & b\\ \overline{b} & c\end{pmatrix} = \begin{pmatrix}
c & b\\ \overline{b} & a\end{pmatrix},\]
i.e., the $Q_\alpha$ are no multiplication operators as in \Cref{cor:TreprUnitalAlegbra}.
\exmsymbol
\end{exm}

While the previous example shows that there exist $\rset$-linear maps on $\Herm_m[x_1,\dots,x_n]$ such that the $Q_\alpha$ in \Cref{lem:DiffSumRepresentationRLinearOp} are not multiplication operators, the next result describes all $\rset$-linear maps on $\Herm_m[x_1,\dots,x_n]$ where all $Q_\alpha$ are multiplication operators.

\begin{prop}
Let $n,m\in\nset$ and let
\[T:\Herm_m[x_1,\dots,x_n]\to\Herm_m[x_1,\dots,x_n]\]
be given by
\[T = \sum_{\alpha\in\nset_0^n} \frac{1}{\alpha!}\cdot q_\alpha\cdot\partial^\alpha\]
with $q_\alpha\in\cset^{m\times m}[x_1,\dots,x_n]$ for all $\alpha\in\nset_0^n$.
Then
\[q_\alpha = \id\cdot p_\alpha\]
with $p_\alpha\in\rset[x_1,\dots,x_n]$ for all $\alpha\in\nset_0^n$ and $\id$ is the $m\times m$-identity matrix.
\end{prop}
\begin{proof}
We prove the statement via induction over $d:=|\alpha|\in\nset_0$.
Note, if $A\in\Herm_m$ and $AB\in\Herm_m$ for all $B\in\Herm_m$, then $A = c\cdot\id$ with $c\in\rset$.

\underline{$d=0$:}
Since $TB = q_0B \in\Herm_m[x_1,\dots,x_n]$ for all $B\in\Herm_m$, $q_0 = \id\cdot p_0$ with $p_0\in\rset[x_1,\dots,x_n]$.

\underline{$d\to d+1$:}
Let $\beta\in\nset_0^n$ with $|\beta|=d+1$.
Since
\[q_\beta B = \left(T - \sum_{\alpha\in\nset_0^n:|\alpha|\leq d} \frac{1}{\alpha!}\cdot q_\alpha\cdot\partial^\alpha \right)\!\!\big(Bx^\beta\big) \quad\in \Herm_m[x_1,\dots,x_n]\]
for all $B\in\Herm_m$, again $q_\beta = \id\cdot p_\beta$ for some $p_\beta\in\rset[x_1,\dots,x_n]$.
\end{proof}

\section{$\rset^n$-Positivity Preservers on Matrix Polynomials}
\label{sec:BorceaCharacterizationOnRn}

Before we can state and prove the main theorem of this section, we need the following.

\begin{dfn}\label{dfn:TyM}
Let $n,m\in\nset$ and
\[
	T = \sum_{\alpha \in \nset_0^n} \frac{1}{\alpha!}\cdot (Q_\alpha \times \partial^\alpha) 
\]
with linear $Q_\alpha: \Herm_m \to \Herm_m[x_1, \dots, x_n]$ for all $\alpha \in \nset_0^n$.
For $y\in\rset^n$, we define
\[T_y: \Herm_m[x_1, \dots, x_n] \to \Herm_m[x_1, \dots, x_n]\]
by
\[
	T_y := \sum_{\alpha \in \nset_0^n} \frac{1}{\alpha!} (Q_\alpha(~\cdot~)(y) \times \partial^\alpha),
\]
i.e., $(Tp)(y) = (T_yp)(y)$ for all $p \in \Herm_m[x_1, \dots, x_n]$ and $y \in \rset^n$.
Additionally, for $y \in \rset^n$ and $M \in\Herm_m$, we define
\[T_{y, M}: \Herm_m[x_1, \dots, x_n] \to \rset[x_1, \dots, x_n],\quad p\mapsto T_{y, M}p := \skl{T_y p}{M}.\]
\end{dfn}

Now we state the main theorem of this section.
It characterizes $\rset^n$-positivity preservers of matrix polynomials similar to \cite[Theorem 3.1]{borcea11}.
The main difference is that the coefficients $Q_\alpha$ in the canonical representation are not given as moments of a positive matrix-valued measure, but as moments of a finite signed matrix-valued measure constructed from positive matrix-valued measures.

\begin{thm}\label{thm:BorceaCharacterizationPolynomialMatrices}
Let $m, n \in \nset$ and let
\[
T:\Herm_m[x_1,\dots,x_n]\to\Herm_m[x_1,\dots,x_n],\quad T = \sum_{\alpha \in \nset_0^n} \frac{1}{\alpha!}\cdot (Q_\alpha \times \partial^\alpha)
\]
be linear with linear $Q_\alpha:\Herm_m\to\Herm_m[x_1,\dots,x_n]$ for all $\alpha\in\nset_0^n$.
Then the following are equivalent:
\begin{enumerate}[(i)]
\item $T\pos(\rset^n)\subseteq\pos(\rset^n)$.

\item For all $y\in\rset^n$, $T_y\pos(\rset^n)\subseteq\pos(\rset^n)$.

\item For all $y \in \rset^n$ and $M \in \Herm_{m,+}$,
\[T_{y, M}\pos(\rset^n)\subseteq \big\{f\in\rset[x_1,\dots,x_n] \,\big|\, f\geq 0\ \text{on}\ \rset^n\big\}.\]

\item For all $y\in\rset^n$ and $M\in\Herm_m$, there exist matrix valued measures $\mu_{y, M}$ such that
\begin{enumerate}[(a)]
\item if $M\in\Herm_{m,+}$, then $\mu_{y,M}$ is positive,

\item for all $p \in \Herm_m[x_1, \dots, x_n]$,
\[
	(T_{y, M}p)(x) = \int \skl{p(x+t)}{\diff\mu_{y, M}(t)}.
\]
\end{enumerate}
In this case, for all $\beta \in \nset_0^n$, $y \in \rset^n$, and $M \in\Herm_m$, 
\begin{equation}\label{eq:QBetaIntegralMuTildeMeasure}
Q_\beta(M)(y) = \int t^\beta ~\diff\widetilde{\mu}_{y, M}(t) \quad\text{with}\quad \widetilde{\mu}_{y, M} := \sum_{i, j = 1}^m H_{i, j} \cdot \skl{\mu_{y, H_{i, j}}}{M}.
\end{equation}
\end{enumerate}
\end{thm}

\Cref{thm:BorceaCharacterizationPolynomialMatrices} will be proven in the following manner.
The proof of the equivalence (i) $\Leftrightarrow$ (ii) is an adapted version of the $\rset$-case proof in \cite[Thm.\ 3.1]{borcea11}.
The proof of the equivalence (ii) $\Leftrightarrow$ (iii) follows from the characterization of positive semi-definiteness in \Cref{lem:MatrixPositiveIffInnerProductWithAllPositiveMatricesIsPositive}.
Lastly, the equivalence (iii) $\Leftrightarrow$ (iv) will be proven in the same manner as the $\rset$-case in \cite[Thm.\ 3.1]{borcea11}, but replacing Haviland's Theorem with its matrix version \Cref{thm:PositiveLambdaFunctionalExistsPositiveMeasure}.

\begin{proof}
(i) $\Rightarrow$ (ii):
For all $y \in \rset^n$, $T_y$ commutes with the shift operator, i.e.,
\[
	T_y \circ e^{a\cdot\nabla} = e^{a\cdot\nabla} \circ T_y
\]
for all $a \in \rset^n$, and the shift operator $e^{a\cdot\nabla}$ preserves $\rset^n$-positivity.
Consequently, (i) implies
\begin{align*}
	(T_yp)(a)
	&= (T_yp)(y+(a-y))
\\	&= \left( e^{(a-y)\nabla} T_y p \right)(y)
\\	&= \left( T_y e^{(a-y)\nabla} p \right)(y)
\\	&= \left( T e^{(a-y)\nabla} p \right)(y)\;\; \succeq 0
\end{align*}
for all $a, y \in \rset^n$ and $p\in\pos(\rset^n)$. 
Hence, (ii) is proven.

(ii) $\Rightarrow$ (i):
For all $y \in \rset^n$ and $p\in\pos(\rset^n)$, we have $(Tp)(y) = (T_yp)(y) \succeq 0$.

(ii) $\Leftrightarrow$ (iii):
This is \Cref{lem:MatrixPositiveIffInnerProductWithAllPositiveMatricesIsPositive}.

(iii) $\Rightarrow$ (iv):
Let $y \in \rset^n$ and $M \in \Herm_{m,+}$. 
Define
\[L_{y, M}: \Herm_m[x_1, \dots, x_n] \to \rset, \quad p\mapsto L_{y, M}(p) := (T_{y, M}p)(0).\]
Hence, by (iii), $L_{y, M}(p) \geq 0$ for all $p \in \pos(\rset^n)$ and therefore, by \Cref{thm:PositiveLambdaFunctionalExistsPositiveMeasure}, $L_{y,M}$ is a moment functional and there exists a positive matrix-valued measure $\mu_{y, M}$ such that 
\[
L_{y, M}(p) = \int_{\rset^n} \skl{p(t)}{\diff\mu_{y, M}(t)}
\]
for all $p \in \Herm_m[x_1, \dots, x_n]$. 
Additionally,
\[(T_{y, M}p)(x) = \left[ (e^{x\cdot\nabla}T_{y, M})p \right](0) = \left[ (T_{y, M}e^{x\cdot\nabla})p \right](0) = L_{y, M}(e^{x\cdot\nabla}p)\]
for all $x \in \rset^n$ and $p \in \Herm_m[x_1, \dots, x_n]$, i.e.,
\begin{equation}\label{eq:TyMintegral}
(T_{y, M}p)(x) = \int_{\rset^n} \skl{p(x+t)}{\diff\mu_{y, M}(t)}
\end{equation}
for all $x,y \in \rset^n$, $M\in\Herm_{m,+}$, and $p \in \Herm_m[x_1, \dots, x_n]$.
Since $\Herm_{m,+}\subseteq\Herm_m$ is full dimensional and by linearity of $T_{y,M}$ in $M$ (\Cref{dfn:TyM}), for all $y\in\rset^n$ and $M\in\Herm_m$, there exists a matrix-valued measure $\mu_{y,M}$ in equation (\ref{eq:TyMintegral}).
$\mu_{y,M}$ is positive for $M\in\Herm_{m,+}$.

Now let $\beta \in \nset_0^n$ and $x, y \in \rset^n$.
then, by \Cref{lem:IntegralInnerProductSwap},
\begin{align*}
	\skl{\int_{\rset^n} (x+t)^\beta ~\diff\mu_{y, M_1}(t)}{M_2}
	&= \int_{\rset^n} \skl{M_2 \cdot (x+t)^\beta}{~\diff\mu_{y, M_1}(t)}
\\	&= T_{y, M_1}(M_2 \cdot x^\beta)
\\	&= \skl{T_y(M_2 \cdot x^\beta)}{M_1}
\end{align*}
for all $M_1, M_2 \in\Herm_m$.
Hence,
\begin{align*}
	\sum_{i, j = 1}^m H_{i, j} \cdot \skl{\int_{\rset^n} (x+t)^\beta ~\diff\mu_{y, H_{i, j}}(t)}{M}
	&= \sum_{i, j = 1}^m H_{i, j} \cdot \skl{T_y(M \cdot x^\beta)}{H_{i, j}}
\\	&= T_y(M \cdot x^\beta)
\\	&= \sum_{\alpha \preceq \beta} \binom{\beta}{\alpha}\cdot Q_\alpha(M)(y)\cdot x^{\beta-\alpha}
\end{align*}
for all $M \in\Herm_m$. 
With $x = 0$, it follows that
\begin{align*}
	Q_\beta(M)(y)
	&= \sum_{i, j = 1}^m H_{i, j} \cdot \skl{\int_{\rset^n} t^\beta ~\diff\mu_{y, H_{i, j}}(t)}{M}
= \int_{\rset^n} t^\beta ~\diff\widetilde{\mu}_{y, M}(t)
\end{align*}
for all $\beta \in \nset_0^n$, $y \in \rset^n$, and $M \in\Herm_m$ with $\widetilde{\mu}_{y, M}$ from (\ref{eq:QBetaIntegralMuTildeMeasure}).

(iv) $\Rightarrow$ (iii): 
Let $y \in \rset^n$ and $M \in\Herm_{m,+}$.
By (iv), there exists a positive matrix-valued measure $\mu_{y, M}$ with 
\[
	(T_{y, M}p)(x) = \int_{\rset^n} \skl{p(x+t)}{\diff\mu_{y, M}(t)}
\]
for all $p \in\Herm_m[x_1, \dots, x_n]$.
Let $p \in \pos(\rset^n)$.
Then, by \Cref{rem:PositiveMatrixValuedMeasureRadonNikodymDerivative},
\[(T_{y, M}p)(x) \geq 0\]
for all $y\in\rset^n$ and $M\in\Herm_{m,+}$ which proves (iii).
\end{proof}

The condition (iv) of \Cref{thm:BorceaCharacterizationPolynomialMatrices} allows $Q_\alpha$ to be written as moments of a measure $\widetilde{\mu}_{y, M}$, i.e., 
\[
	Q_\beta(M)(y) = \int_{\rset^n} t^\beta ~\diff\widetilde{\mu}_{y, M}(t)
\]
for $\beta \in \nset_0^n$, $M \in\Herm_m$, and $y \in \rset^n$.
The measure $\widetilde{\mu}_{y, M}$ on $(\rset^n,\fB(\rset^n))$ is a finite signed matrix-valued measure but was not proven to be positive.
Problems can arise due to indeterminacy of the underlying moment problems.

\begin{exm}\label{exm:QAlphaMeasureNotPositive}
Let
\[
	f(x) := \frac{1}{\sqrt{2\pi}}\cdot \one_{(0, \infty)}(x)\cdot \frac{1}{x}\cdot e^{-\frac{1}{2}\log(x)^2}
\]
be the density of the $\log$-normal distribution.
Then
\[\int_0^\infty x^n\cdot f(x) ~\diff x = e^{\frac{1}{2}n^2}\]
and
\[\int_0^\infty x^n\cdot \sin(2\pi\log(x))\cdot f(x) ~\diff x = 0\]
for all $n \in \nset_0$.
The measures $\mu_{y, M}$ in \Cref{thm:BorceaCharacterizationPolynomialMatrices} (iv) are not required to fulfill any linearity conditions like
\[
\mu_{y, A+B} = \mu_{y, A} + \mu_{y, B} \quad\text{or}\quad \mu_{y, \lambda A} = \lambda \mu_{y, A}
\]
for $A, B \in\Herm_m$ and $\lambda\in\rset$.
The values of these measures can differ on non-compact sets $K \subseteq \rset$, where the polynomials can not approximate the indicator functions $\one_K$.
The proof (iii) $\Rightarrow$ (iv) of \Cref{thm:BorceaCharacterizationPolynomialMatrices} can not preserve such information in cases of indeterminacy.
Hence, let $m = 2$ and define
\[
	\diff\mu_{y, M}(x) := M f(x) ~\diff x
\]
for $y \in \rset$ and $M \in\Herm_2$.
The measure $\mu_{y, M}$ is a positive matrix-valued measure for all $M \in \Herm_{m,+}$.
Let  
\[
	X := H_{1, 2} = \frac{1}{\sqrt{2}}\begin{pmatrix}
		0 & 1
	\\	1 & 0
	\end{pmatrix}
\]
and define 
\[
\nu_{y, M} := \begin{cases}
		\mu_{y, M} &\text{for}\ M \neq X,
	\\	\mu_{y, X} + X \cdot 10\sin(2\pi\log x )\cdot f(x) ~\diff x &\text{for}\ M = X
	\end{cases}
\]
for all $y \in \rset$ and $M \in\Herm_2$. 

Then $\mu_{y, M}$ and $\nu_{y, M}$ are both positive matrix-valued measures for $M \in\Herm_{m,+}$ and $y \in \rset$.
Both measures give rise to the same operators $T_{y, M}$ for $y \in \rset$ and $M \in\Herm_2$, but
\[
	\widetilde{\nu}_{y, M} := \sum_{i, j = 1}^m H_{i, j} \cdot \skl{\nu_{y, H_{i, j}}}{M}
\]
is not a positive matrix-valued measure for all $M \in \Herm_{2,+}$. 
This is because, for
\[Z := \begin{pmatrix}
		1 & 1
	\\	1 & 1
	\end{pmatrix} \in \Herm_{m,+},\]
it holds that 
\begin{align*}
	\diff\widetilde{\nu}_{y, Z}(x)
	&= \begin{pmatrix}
		1 & 1 + 10\cdot\sin(2\pi\log x )
	\\	1 + 10\cdot\sin(2\pi\log x) & 1
	\end{pmatrix}\cdot f(x) ~\diff x
\end{align*}
where
\[
	\widetilde{\nu}_{y, Z}((1, \infty)) \in \Herm_2 \setminus \Herm_{2,+}
\]
is not a positive semi-definite matrix, despite $Z \in \Herm_{2,+}$. 
\exmsymbol
\end{exm}

Hence, through indeterminacy, the measures in condition (iv) of \Cref{thm:BorceaCharacterizationPolynomialMatrices} can be chosen such that $Q_\beta(M)(y)$ are the moments of a matrix-valued measure $\widetilde{\nu}_{y, M}$ which is not positive for all $M \in \Herm_{m,+}$.  
On the other hand, in \Cref{exm:QAlphaMeasureNotPositive} choosing
\[
	\widetilde{\mu}_{y, M} := \sum_{i, j = 1}^m H_{i, j} \cdot \skl{\mu_{y, H_{i, j}}}{M}
\]
is a positive matrix-valued measure for all $M \in\Herm_{m,+}$ with the same moments as $\widetilde{\nu}_{y, M}$. 
Hence, the $Q_\beta(M)(y)$ can be given as moments of the positive matrix-valued measure $\widetilde{\mu}_{y, M}$ for $y \in \rset$ and $M \in\Herm_{m,+}$.

We were not able to prove that there is always a choice, such that the $Q_\beta$ can be given by a positive matrix-valued measure.
These problems do not arise for positivity preservers on compact sets, as moment sequences on compact sets are always determinate.

The difference in \Cref{thm:BorceaCharacterizationKPositivePreserverPolynomialMatrices} with $m\geq 2$ to the case of real coefficients ($m=1$) lies in being able to choose different measures with the same moments for different basis vectors of $\Herm_m$. 
In the case $m=1$, there is only one basis vector to choose a measure for and thus this problem does not arise.

For $m=1$, \Cref{thm:BorceaCharacterizationPolynomialMatrices} reduces to \cite[Thm.\ 3.1]{borcea11}.

\begin{cor}[{\cite[Thm.\ 3.1]{borcea11}}]\label{cor:BorceaTheoremHerm1}
Let $n \in \nset$ and let 
\[
T:\rset[x_1,\dots,x_n]\to\rset[x_1,\dots,x_n],\quad T = \sum_{\alpha \in \nset_0^n} \frac{1}{\alpha!}\cdot q_\alpha\cdot \partial^\alpha
\]
be linear with polynomials $q_\alpha \in \rset[x_1, \dots, x_n]$ for all $\alpha \in \nset_0^n$. 
Let
\[
	\rset[x_1, \dots, x_n]_+ := \{p \in \rset[x_1, \dots, x_n] \mid p \geq 0\ \text{on}\ \rset^n\}
\]
Then the following are equivalent: 
\begin{enumerate}[(i)]
\item $T\rset[x_1, \dots, x_n]_+ \subseteq \rset[x_1, \dots, x_n]_+$.

\item For every $y \in \rset^n$, $T_y\rset[x_1, \dots, x_n]_+ \subseteq \rset[x_1, \dots, x_n]_+$ where 
\[
T_y = \sum_{\alpha \in \nset_0^n} \frac{1}{\alpha!}\cdot q_\alpha(y)\cdot \partial^\alpha
\]
for $y \in \rset^n$.

\item For every $y \in \rset^n$, there exists a measure $\mu_y$ such that 
\[
(T_yp)(x) = \int_{\rset^n} p(x+t) ~\diff\mu_y(t)
\]
for all $p \in \rset[x_1, \dots, x_n]$.
In this case, for $\beta \in \nset_0^n$ and $y \in \rset^n$, 
\begin{equation}\label{eq:coeffsM=1}
q_\beta(y) = \int_{\rset^n} t^\beta ~\diff\mu_y(t).
\end{equation}
\end{enumerate}
\end{cor}
\begin{proof}
For $m=1$, $\Herm_1 = \rset$, $\Herm_{1,+} = \{\lambda \in \rset \mid \lambda \geq 0\}$, and $\Pos(\rset^n) = \rset[x_1, \dots, x_n]_+$. 
Additionally, the inner product
\[\skl{\,\cdot\,}{\,\cdot\,}: \Herm_1 \times \Herm_1 \to \rset\]
reduces to
\[\skl{\lambda_1}{\lambda_2} = \lambda_1 \cdot \lambda_2\]
for all $\lambda_1, \lambda_2 \in \Herm_1 = \rset$. 
For $\alpha\in\nset_0^n$, every map
\[Q_\alpha: \Herm_1 \to \rset[x_1, \dots, x_n]\]
is uniquely determined by its value $q_\alpha := Q_\alpha(1) \in \rset[x_1, \dots, x_n]$.
Hence,
\[
	T = \sum_{\alpha\in\nset_0^n} \frac{1}{\alpha!}\cdot (Q_\alpha \times \partial^\alpha) = \sum_{\alpha \in \nset_0^n} \frac{1}{\alpha!}\cdot q_\alpha\cdot \partial^\alpha.
\] 
For $y \in \rset^n$, equivalently
\[T_y = \sum_{\alpha\in\nset_0^n} \frac{1}{\alpha!}\cdot Q_\alpha(1)(y)\cdot\partial^\alpha = \sum_{\alpha\in\nset_0^n} \frac{1}{\alpha!}\cdot q_\alpha(y)\cdot\partial^\alpha.\]
By $\Pos(\rset^n) = \rset[x_1, \dots, x_n]_+$, \Cref{thm:BorceaCharacterizationPolynomialMatrices} gives (i) $\Leftrightarrow$ (ii). 

(i) $\Rightarrow$ (iii): 
By \Cref{thm:BorceaCharacterizationPolynomialMatrices}, for $y \in \rset^n$ and $\lambda \in \rset$, there exist measures $\mu_{y, \lambda}$ such that
\begin{enumerate}[(a)]
\item for all $y \in \rset^n$ and $\lambda \geq 0$, $\mu_{y, \lambda}$ is a positive measure and,

\item for all $x, y \in \rset^n$ and $\lambda \in \rset$,
\[
T_{y, \lambda}p(x) = \int_{\rset^n} \skl{p(x+t)}{\diff\mu_{y, \lambda}(t)}
\]
for all $p \in \rset[x_1, \dots, x_n]$.
\end{enumerate}
Since the inner product reduces to just multiplication, it follows, that $T_{y, \lambda} = \lambda T_y$ for all $y \in \rset^n, \lambda \in \rset$. 
Additionally the integration reduces to
\[
	\int_{\rset^n} \skl{p(t)}{\diff\mu_{y, \lambda}(t)} = \int_{\rset^n} p(t) ~\diff\mu_{y, \lambda}(t)
\]
for all $y \in \rset^n, \lambda \in \rset$, and $p \in \rset[x_1, \dots, x_n]$.
Hence,
\[
	T_yp(x) = T_{y, 1}p(x) = \int_{\rset^n} \skl{p(x+t)}{\diff\mu_{y, 1}(t)} = \int_{\rset^n} p(x+t) ~\diff\mu_{y, 1}(t)
\]
for all $x, y \in \rset^n$ and $p \in \rset[x_1, \dots, x_n]$ which proves (iii). 

The statement (\ref{eq:coeffsM=1}) for the polynomial coefficients $q_\alpha$, $\alpha \in \nset_0^n$, follows from the fact, that the sum in (\ref{eq:QBetaIntegralMuTildeMeasure}) from \Cref{thm:BorceaCharacterizationPolynomialMatrices} reduces to a single element due to having a $1$-dimensional real vector space $\Herm_1 = \rset$.

(iii) $\Rightarrow$ (i):
Define $\mu_{y, \lambda} := \lambda \mu_y$.
Then (i) follows from \Cref{thm:BorceaCharacterizationPolynomialMatrices}.
\end{proof}

\section{Matrix $K$-Positivity Preserver for Compact $K$}
\label{sec:BorceaCharacterizationOnCompactSets}

Contrary to the previous section were we worked with collections of matrix-valued measures
\[(\mu_M)_{M \in \Herm_m}\]
in \Cref{thm:BorceaCharacterizationPolynomialMatrices} (iv), we work now with operator-valued measures
\[\mu: \cA \to \cL(\Herm_m, \Herm_m),\]
where $\cL(\Herm_m, \Herm_m)$ denotes the space of linear operators $\Herm_m \to \Herm_m$. 
For a given matrix $M \in \Herm_m$, we can write $\mu_M$ as the measure
\[\mu(\,\cdot\,)(M): \cA \to \Herm_m.\]
The difference is that collections of matrix-valued measures do not need to fulfill any linearity conditions, while operator-valued measures do.
This solves the problem outlined in \Cref{exm:QAlphaMeasureNotPositive}, but the existence of such operator-valued measures can only be proven so far only on compact sets $K \subseteq \rset^n$.

\begin{dfn}[see \protect{\cite[Sect.\ 2]{cimpzalar13}}]
Let $m \in \nset$ and let $\cA$ be a $\sigma$-algebra.
A set function
\[\mu: \cA \to \cL(\Herm_m, \Herm_m)\]
is called a \emph{positive operator-valued measure} if
\[\mu[M] := \mu(\,\cdot\,)(M): \cA \to \Herm_m\]
is a positive matrix-valued measure for all $M\in\Herm_{m,+}$.
\end{dfn}

Integrals with respect to such measures are defined in the following way.

\begin{dfn}
Let $m, n \in \nset$, $K \subseteq \rset^n$ be closed, $p \in \Herm_m[x_1, \dots, x_n]$, and $\mu: \fB(K) \to \cL(\Herm_m, \Herm_m)$ be a positive operator-valued measure.
We define
\[\int_K p(x) ~\diff\mu(x) := \sum_{i, j = 1}^m \int_K \skl{p(x)}{H_{i, j}} ~\diff\mu[H_{i, j}](x)\quad \in \Herm_m.\]
\end{dfn}

Integrals of positive matrix polynomials with respect to positive operator-valued measures are again positive due to the linearity conditions on operator-valued measures. 
Therefore, the following holds.

\begin{lem}\label{lem:IntegralPositiveOperatorValuedMeasurePositivePolynomialPositive}
Let $m, n \in \nset$, let $K \subseteq \rset^n$ be closed,
\[\mu: \fB(K) \to \cL(\Herm_m, \Herm_m)\]
be a positive operator-valued measure, and let $p \in \pos(K)$.
Then
\[\int_K p(x) ~\diff\mu(x) \succeq 0.\]
\end{lem}
\begin{proof}
Let $X \in \Herm_{m,+}$ and define the matrix-valued measure
\[\widetilde{\mu}_X: \fB(K) \to \Herm_m \quad\text{by}\quad \widetilde{\mu}_X := \sum_{i, j = 1}^m \skl{\mu[H_{i, j}]}{X} H_{i, j}.\]
Let $M \in \Herm_{m,+}$ and $S \in \fB(K)$.
By linearity of operator-valued measures,
\begin{align*}
\skl{\widetilde{\mu}_X(S)}{M}
	&= \sum_{i, j = 1}^m \skl{\mu(S)(H_{i, j})}{X} \skl{M}{H_{i, j}}
\\	&= \skl{\mu(S)\left(\sum_{i, j = 1}^m  \skl{M}{H_{i, j}}H_{i, j}\right)}{X} 
= \skl{\mu(S)(M)}{X}
\geq 0.
\end{align*}
Hence, $\widetilde{\mu}_X$ is a positive matrix-valued measure by \Cref{lem:MatrixPositiveIffInnerProductWithAllPositiveMatricesIsPositive} and
\begin{align*}
	\skl{\int_K p(x) ~\diff\mu(x)}{X}
	&= \sum_{i, j = 1}^m \int_K \skl{p(x)}{H_{i, j}} ~\diff\skl{\mu[H_{i, j}]}{X}(x)
\\	&= \sum_{i, j = 1}^m \int_K \skl{p(x)}{H_{i, j}} ~\diff\skl{\widetilde{\mu}_X}{H_{i, j}}(x)
\\ &= \int_K \skl{p(x)}{\diff\widetilde{\mu}_X(x)} \quad\geq 0
\end{align*}
for all $p \in \Pos(K)$, where the positive semi-definiteness follows from \Cref{rem:PositiveMatrixValuedMeasureRadonNikodymDerivative}.
Since $X\in\Herm_{m,+}$ is arbitrary and by \Cref{lem:MatrixPositiveIffInnerProductWithAllPositiveMatricesIsPositive}, the assertion is proven.
\end{proof}

This positivity from \Cref{lem:IntegralPositiveOperatorValuedMeasurePositivePolynomialPositive} can fail without the linear conditions of operator-valued measures.
That is, if we would only look at collections of measures as in \Cref{thm:BorceaCharacterizationPolynomialMatrices} (iv), then this positivity will not hold anymore, as shown in \Cref{exm:QAlphaMeasureNotPositive}.
For the characterization of operator-valued measures according to \Cref{thm:BorceaCharacterizationPolynomialMatrices}, we use \Cref{thm:PositiveMapCompactSetExistsPositiveMeasure}.

Note, $K \subseteq \rset^n$ does not need to be a compact set for \Cref{thm:PositiveMapCompactSetExistsPositiveMeasure} to hold. 
J.\ Cimprič and A.\ Zalar use a more general setting on arbitrary closed $K \subseteq \rset^n$, which requires additional conditions on the functional $L$.
This is because the proof of \cite[Thm.\ 4]{cimpzalar13} uses Arveson's Extension Lemma, which replaces simple positivity with complete positivity.
This is necessary, since there are positive but not completely positive functionals which can fail to have a linear extension for some finite-dimensional Hilbert spaces, see e.g.\ \cite{chiribella23}.

However, on compact sets $K$ this additional condition is not required anymore and we can formulate the following theorem.

\begin{thm}\label{thm:BorceaCharacterizationKPositivePreserverPolynomialMatrices}
Let $m, n \in \nset$, let $K \subseteq \rset^n$ be compact, and let
\[T: \Herm_m[x_1, \dots, x_n] \to \Herm_m[x_1, \dots , x_n]\]
be a linear with canonical representation
\[
	T = \sum_{\alpha \in \nset_0^n} \frac{1}{\alpha!}\cdot (Q_\alpha \times \partial^\alpha).
\]
Then the following are equivalent:
\begin{enumerate}[(i)]
\item $T$ is a $K$-positivity preserver, i.e., $T\pos(K) \succeq \pos(K)$.

\item For each $y \in K$, there exists a positive operator-valued measure $\nu_y$ on ${K-y} \subseteq \rset^n$ such that 
\[
	Q_\beta(M)(y) = \int_{K-y} t^\beta ~\diff\nu_y[M](t)
\]
for all $\beta \in \nset_0^n$.
In this case,
\[
	(T_y p)(x) = \int_{K-y} p(x+t) ~\diff\nu_y(t)
\]
for all $x \in K$ and $p \in \Herm_m[x_1, \dots, x_n]$.
\end{enumerate}
\end{thm}

The proof is an adapted version of the proof of \cite[Thm.\ 3.5]{didio25KPosPresGen} using \cite[Thm.\ 4]{cimpzalar13} on compact sets.

\begin{proof}
(i) $\Rightarrow$ (ii): 
Let $y \in K$ and define
\[L_y: \Herm_m[x_1, \dots, x_n] \to \Herm_m,\quad p\mapsto L_y(p) := (Tp)(y).\]
Then $L_y(p) \succeq 0$ for all $p \in \Pos(K)$.
\Cref{thm:PositiveMapCompactSetExistsPositiveMeasure} asserts the existence of a positive Borel operator-valued measure $\mu_y: \fB(K) \to \cL(\Herm_m, \Herm_m)$ such that
\[L_y(p) = \int_K p(x) ~\diff \mu_y(x)\]
for all $p \in \Herm_m[x_1, \dots, x_n]$. 
Additionally,
\[L_y(M \cdot p) = (T(M \cdot p))(y) = (T_y(M \cdot p))(y) = \sum_{\alpha \in \nset_0^n} \frac{1}{\alpha!}\cdot Q_\alpha(M)(y)\cdot (\partial^\alpha p)(y)\]
for all $M \in \Herm_m$ and $p \in \rset[x_1, \dots, x_n]$.
Let $M \in \Herm_m$ and $\beta \in \nset_0^n$.
Set $\widetilde{p}(x) := (x-y)^\beta$.
Then
\[
\partial^\alpha \widetilde{p}(y) = \begin{cases}
0 & \alpha \neq \beta
\\	\beta! & \alpha = \beta
\end{cases}\]
for all $\alpha\in\nset_0^n$ and, hence,
\[
L_y(M \cdot \widetilde{p}) 
= Q_\beta(M)(y) 
= \int_K M\cdot (x-y)^\beta ~\diff\mu_y(x) 
= \int_K (x-y)^\beta ~\diff\mu_y[M](x).
\]
for all $M \in \Herm_m$. 
Define $\nu_y(\,\cdot\,) := \mu_y(\, \cdot\, + y)$ as the pushforward measure regarding to the translation in $\rset^n$.
Then
\[\nu_y: \fB(K) \to \cL(\Herm_m, \Herm_m)\]
is a positive Borel operator-valued measure and 
\[
	Q_\beta(M)(y) = \int_{K-y} t^\beta ~\diff\nu_y[M](t)
\]
for all $\beta \in \nset_0^n$, $y \in K$, and $M \in \Herm_m$. 
Furthermore, for $y \in \rset^n$, $M \in \Herm_m$, and $\beta \in \nset_0^n$, it follows from the multi-binomial theorem that
\begin{align*}
	T_y(Mx^\beta)
	&= \sum_{\alpha \in \nset_0^n} \binom{\beta}{\alpha}\cdot Q_\alpha(M)(y)\cdot x^{\beta-\alpha}
\\	&= \sum_{\alpha \in \nset_0^n} \binom{\beta}{\alpha}\cdot x^{\beta-\alpha}\cdot \int_{K-y} t^\alpha ~\diff\nu_y[M](t)
\\	&= \int_{K-y} (x+t)^\beta ~\diff\nu_y[M](t)
\\	&= \int_{K-y} M\cdot (x+t)^\beta ~\diff\nu_y(t)
\end{align*}
and, hence, by linearity,
\[
(T_yp)(x) = \int_{K-y} p(x+t) ~\diff\nu_y(t)
\]
for all $x, y \in \rset^n$ and all $p \in \Herm_m[x_1, \dots, x_n]$.

(ii) $\Rightarrow$ (i):
Let $y \in K$ and let $\nu_y$ be a positive matrix-valued measure on $K-y$ such that
\[
	\int_{K-y} t^\beta ~\diff\nu_y[M](t) = Q_\beta(M)(y)
\]
for all $M \in \Herm_m$ and all $\beta\in\nset_0^n$.
Then
\begin{align*}
\int_{K-y} M\cdot (t+y)^\beta ~\diff\nu_y(t)
	&= \sum_{\alpha \preceq \beta} \binom{\beta}{\alpha}\cdot y^{\beta-\alpha}\cdot \int_{K-y} t^\alpha ~\diff\nu_y[M](t)
\\	&= \sum_{\alpha \in \nset_0^n} \frac{1}{\alpha!}\cdot Q_\alpha(M)(y)\cdot \partial^\alpha y^\beta
\\	&= (T_y(M \cdot x^\beta))(y)
\\	&= (T(M \cdot x^\beta))(y)
\end{align*}
for all $\beta \in \nset_0^n$, $M \in \Herm_m$, and $y \in K$.
By linearity,
\[
	\int_{K-y} p(t+y) ~\diff\nu_y(t) = (Tp)(y)
\]
for all $p \in \Herm_m[x_1, \dots, x_n]$ and $y \in K$.
Define $\mu_y(\,\cdot\,) := \nu_y(\,\cdot\, - y)$ as the pushforward measure.
Then, by \Cref{lem:IntegralPositiveOperatorValuedMeasurePositivePolynomialPositive},
\[
	(Tp)(y) = \int_K p(t) ~\diff\mu_y(t) \succeq 0
\]
for all $p \in \Pos(K)$ and $y \in K$.
Hence, $T$ is a $K$-positivity preserver.
\end{proof}


\section*{Funding}

The authors are supported by the Deutsche Forschungs\-gemein\-schaft DFG with the grant DI-2780/2-1 and the research fellowship of the first author at the Zukunfts\-kolleg of the University of Konstanz, funded as part of the Excellence Strategy of the German Federal and State Government.

\section*{Conflict of Interest Statement}

The author declares no conflict of interest.

\section*{Data Availability Statement}

There is no associated data.


\providecommand{\bysame}{\leavevmode\hbox to3em{\hrulefill}\thinspace}
\providecommand{\MR}{\relax\ifhmode\unskip\space\fi MR }
\providecommand{\MRhref}[2]{%
  \href{http://www.ams.org/mathscinet-getitem?mr=#1}{#2}
}
\providecommand{\href}[2]{#2}

\end{document}